\documentclass[11pt]{article}
\usepackage{graphicx} 
\usepackage[top=27mm, bottom=27mm, left=25mm, right=25mm]{geometry}
\usepackage{amsfonts}
\usepackage{amssymb}
\usepackage{amsthm}
\usepackage{verbatim}
\usepackage[T1]{fontenc}
\usepackage{graphicx}
\usepackage{diagbox}
\usepackage{xcolor}
\usepackage{amsmath}
\usepackage{array,multirow}
\usepackage{dsfont}
\usepackage{mathcomp}
\usepackage{thmtools}
\usepackage{thm-restate}
\usepackage{enumerate}
\usepackage{bm}
\usepackage{ulem}
\usepackage [latin1]{inputenc}
\usepackage[pagebackref]{hyperref}
\usepackage[capitalize]{cleveref}
\hypersetup{
	colorlinks,
	linkcolor={blue!60!black},
	citecolor={green!60!black},
	urlcolor={green!60!black}
}

\usepackage{xcolor}
\definecolor{cream}{RGB}{203, 237, 204}

\newtheorem{theorem}{Theorem}[section]
\newtheorem{lemma}[theorem]{Lemma}
\newtheorem{proposition}[theorem]{Proposition}
\newtheorem{corollary}[theorem]{Corollary}

\newtheorem{definition}[theorem]{Definition}

\newtheorem{claim}[theorem]{Claim}

\newtheorem{observation}[theorem]{Observation}
\newcommand{\F}{{\mathcal{F}}}
\newcommand{\mG}{{\mathcal{G}}}

\newcommand{\C}{{\mathcal{C}}}
\newcommand{\T}{{\mathcal{T}}}

\newcommand{\mH}{{\mathcal{H}}}
\newcommand{\mF}{{\mathcal{F}}}

\newcommand{\mS}{{\mathcal{S}}}

\newcommand{\keywords}[1]{\par\noindent\textbf{Keywords:} #1}

\title{Quantitative frameproof codes and hypergraphs}

\author{
Wenjie~Zhong\thanks{School of Mathematical Sciences, University of Science and Technology of China, Hefei, 230026, Anhui, China.
    Emails: \texttt{zhongwj@mail.ustc.edu.cn}.}
\and Xinqi~Huang\thanks{School of Mathematical Sciences, University of Science and Technology of China, Hefei, China and Extremal Combinatorics and Probability Group (ECOPRO), Institute for Basic Science (IBS), Daejeon, South Korea.
    Emails: \texttt{huangxq@mail.ustc.edu.cn}.}
\and Xiande~Zhang\thanks{School of Mathematical Sciences, University of Science and Technology of China, Hefei, 230026, Anhui, China; Hefei National Laboratory, University of Science and Technology of China, Hefei, 230088, Anhui, China.
    Email: \texttt{drzhangx@ustc.edu.cn}.}
}
\date{}

\begin{document}
\maketitle
\begin{abstract}
Frameproof codes are a class of secure codes introduced by Boneh and
Shaw in the context of digital fingerprinting, and have been widely studied from a combinatorial point of view.
In this paper, we study a quantitative extension of frameproof codes and hypergraphs, referred to as {\it quantitative frameproof codes and hypergraphs}. We give asymptotically optimal bounds on the maximum sizes of   these structures and determine their exact sizes for a broad range of parameters. In particular, we introduce a generalized version of the Erd\H{o}s  matching number in our proof and derive relevant estimates for it.
%

\end{abstract}

\section{Introduction}

The study of maximum cardinalities of combinatorial objects that avoid specific forbidden configurations represents a central theme in extremal combinatorics, with various applications in coding theory and theoretical computer science (see, e.g.~\cite{jukna2011extremal}).


Among such systems, frameproof codes have attracted considerable attention \cite{blackburn2003frameproof,liu2024near,staddon2001combinatorial}.  A frameproof code is a code in which every codeword can be distinguished from any small collection of other codewords by at least one coordinate.
These codes play an important role in several practical scenarios including digital fingerprinting \cite{boneh1998collusion} and   traitor tracing schemes \cite{staddon2001combinatorial}. Moreover, frameproof codes are closely related to several other classes of secure codes,
 such as separable codes \cite{blackburn2015probabilistic,cheng2011separable,gao2014new}, parent-identifying codes \cite{gu2019probabilistic, shangguan2018new} and traceability codes \cite{blackburn2010traceability,gu2017bounds,kabatiansky2019traceability,staddon2001combinatorial}.
 Connections between frameproof codes and other combinatorial objects have also been extensively studied, including separating hash families \cite{bazrafshan2011bounds,blackburn2008bound,ge2019some,shangguan2016separating,stinson2008generalized,wei2025separating} and cover-free families \cite{d2017cover,ge2019some, staddon2001combinatorial}.

The central problem in the study of frameproof codes is to determine the maximum cardinality of such codes for given parameters. Liu et al. \cite{liu2024near,liu2025near} established asymptotically optimal lower bounds that match the upper bounds previously derived by Blackburn
 \cite{blackburn2003frameproof}. In \cite{alon2023near}, Alon and Holzman introduced focal families as a variant of the known sunflower problem \cite{erdos1960intersection,erdHos1978combinatorial}.
More recently, Huang et al. \cite{huang2024focal} demonstrated that the methodology of frameproof codes can be effectively applied to determine asymptotically optimal bounds on the maximum cardinality of focal-free hypergraphs and codes.

In this paper, we introduce a general framework which includes frameproof codes and focal-free codes as two extreme cases (and their hypergraph formulations as well). We call them  {\it quantitative frameproof codes and  hypergraphs}.  In the context of digital fingerprinting (see e.g. \cite{liu2024near}), a frameproof code can be used to protect copyrighted materials from unauthorized use in a scenario where a coalition of  $c$ users with fingerprints $\bm x^1,\ldots,\bm x^c\in \mathcal{C}\subset [q]^n$ can produce a pirate copy by taking some vector $\bm x=(x_1,\ldots,x_n)\in [q]^n$ with each $x_i\in \{x_i^j:j\in [c]\}$. Here $[*]$ denotes the set of consecutive integers $\{1,2,\ldots,*\}$. A set $\mathcal{C}$ of fingerprints is $c$-frameproof if and
only if any coalition of at most $c$ users can not frame another
user, that is, the vector $\bm x$ produced by the coalition must be a member of the coalition. A quantitative frameproof code can be used similarly, but under the assumption that a coalition has less power than that in the above scenario: the same coalition (but allowing repeated users) can only produce a vector $\bm x$  with each $x_i$ appearing at least $s$ times in the multiset $\{\{x_i^j:j\in [c]\}\}$ for some threshold parameter $s$. A frameproof code is a special case when $s=1$. Formal definitions of  quantitative frameproof codes and hypergraphs are given below.

%
%


For any integers $q,c\ge 2$ and $1\le s\le c-1$, a code $\mathcal{C}\subset [q]^n$ is said to be {\it $(c,s)$-frameproof} if for every $c+1$ codewords $\bm x^0,\mathbf{x}^1,\ldots,\bm x^c$, where $\bm x^0\neq \bm x^j$ for each $j\in [c]$ but $\bm x^1,\ldots,\bm x^c$ are not necessarily distinct, there must be some coordinate $i\in [n]$ such that $|\{j\in [c]:x_i^j=x_i^0\}|<s$. Here the quantitative parameter $s$ indicates that to what extent a codeword can be distinguished from other $c$ repeatable codewords.
With respect to $s$, the quantitative frameproof code contains two extreme cases: when $s=1$, $\mathcal{C}$ is a frameproof code  \cite{liu2024near}; when $s=c-1$, $\mathcal{C}$ turns to be a focal-free code  \cite{huang2024focal}.

When $q=2$, we also consider a one-sided counterpart of the quantitative frameproof code, that is the quantitative frameproof hypergraph.
For a set $A$, we denote $2^A$ to be the power set of $A$ and $\binom{A}{k}$ to be the set of all $k$-subsets of $A$.
A  hypergraph $\mathcal{F}\subseteq 2^{[n]}$  is said to be {\it $(c,s)$-frameproof} if for every $c+1$ edges $A_0,A_1,\ldots,A_c$, where $A_0\neq A_j$ for each $j\in [c]$ but $A_1,\ldots,A_c$ are not necessarily distinct, there must be some vertex $i\in A_0$ such that $|\{j\in [c]:i\in A_j\}|<s$. Note that when $s=1$, we come to the definition of cover-free families \cite{wei2023cover}; when $s=c-1$, $\mathcal{F}$ becomes a focal-free hypergraph \cite{huang2024focal}.

\subsection{Main results}

Let $f_{c,s}(n,k)$ denote the maximum size of a $(c,s)$-frameproof hypergraph in $\binom{[n]}{k}$. Let $f_{c,s}^{q}(n)$ denote the maximum size of a $(c,s)$-frameproof code in $[q]^n$.
Our main results establish asymptotically optimal estimates for both
 $f_{c,s}(n,k)$ for large $n$ and $f_{c,s}^{q}(n)$ for large $q$. A key ingredient in our analysis is the generalized Erd\H{o}s matching number $m(n,t,\lambda;k_1,k_2)$, which will be formally defined below the two theorems.


\begin{theorem}\label{mainthm-turan density}
Let $c$, $s$ and $k$ be fixed integers satisfying $c,k\ge 2$ and $1\le s \le c-1$.
We have
    \begin{align*}     \lim_{n\to\infty}f_{c,s}(n,k)\bigg/\binom{n}{t}
    =
    \frac{1}
    {\binom{k}{t}- m(k,t,\lambda; s+1,c-s+1)},
    \end{align*}
    where $t=\lceil sk/c\rceil$ and $\lambda\in [c]$ is the integer that satisfies $\lambda\equiv sk \pmod{c}$.
\end{theorem}

\begin{theorem}\label{mainthm:code}
	For any fixed integers $n$, $c$ and $s$ satisfying $n,c\geq 2$ and $1\le s\le c-1$, we have
\begin{align*}
    \lim\limits_{q\to\infty} f_{c,s}^{q}(n) \bigg/ q^t
	=
 \frac{\binom{n}{t}}{\binom{n}{t}-m(n,t,\lambda;s+1,c-s+1)},
\end{align*}
where $t=\lceil sn/c\rceil$ and $\lambda\in [c]$ is the integer that satisfies $\lambda\equiv sn \pmod{c}$.
\end{theorem}

We remark that when $s=1$, the above results come to the asymptotically optimal values for frameproof codes in \cite{liu2024near}; when $s=c-1$, the above results come to the asymptotically optimal values for focal-free hypergraphs and codes in \cite{huang2024focal}.
As you see, our results rely on a newly-defined function $m(n,t,\lambda;k_1,k_2)$, which we call the {\it generalized matching number} defined below. It is a generalization of the known Erd\H{o}s matching number (see, e.g. \cite{erdos1965problem}), denoted by $m(n,t,\lambda)$, which is the maximum number of edges in a $t$-uniform hypergraph on $n$ vertices that does not contain $\lambda$ pairwise disjoint edges. 

\begin{definition}\label{def-disjoint}
	Given positive integers $n,\lambda,k_1,k_2$, we say a collection $A_1, A_2, \ldots, A_\lambda$ of repeatable subsets of $[n]$ is \textbf{$k_1$-disjoint} if for any $k_1$-subset $B\in \binom{[\lambda]}{k_1}$, we have $\cap_{i\in B} A_i = \emptyset$; it is
 \textbf{$k_2$-covering} if for any $k_2$-subset $B\in \binom{[\lambda]}{k_2}$, we have $\cup_{i\in B} A_i = [n]$, or equivalently $\cap_{i\in B} ([n]\backslash A_i) = \emptyset$.

	We say a collection $A_1, A_2, \ldots, A_\lambda$ of repeatable subsets of $[n]$ is \textbf{$(k_1,k_2)$-disjoint} if it is $k_1$-disjoint and $k_2$-covering.
\end{definition}

 Equivalently,  a collection above is $(k_1,k_2)$-disjoint if and only if each point of $[n]$ appears in at least $(\lambda-k_2+1)$ sets and at most $(k_1-1)$ sets, which requires that $\min\{k_1,k_2\}\leq \lambda \leq k_1+k_2-2$.

\begin{definition}[Generalized matching number]\label{def-generalized matching number}
	For positive integers $n, t, \lambda, k_1, k_2$, the generalized matching number $m(n,t,\lambda;k_1,k_2)$ is defined to be the maximum size of a family $\mF\subset \binom{[n]}{t}$ such that $\mF$ contains no $(k_1,k_2)$-disjoint collection of size $\lambda$.
\end{definition}

Note that in Definition~\ref{def-disjoint}, when $k_1> \lambda$, ``$k_1$-disjoint'' makes no sense,  thus ``$(k_1,k_2)$-disjoint'' reduces to ``$k_2$-covering''; when $k_2> \lambda$, it reduces to ``$k_1$-disjoint''.
Therefore, under the conditions in Theorems~\ref{mainthm-turan density} and~\ref{mainthm:code}, when $s=1$ or $s=c-1$, we easily get
$m(n,t,\lambda;s+1,c-s+1)=m(n,t,\lambda)$ ($m(k,t,\lambda;s+1,c-s+1)=m(k,t,\lambda)$). Indeed, when $s=1$ (the situation $s=c-1$ is similar), we have $m(n,t,\lambda;s+1,c-s+1)=m(n,t,\lambda;2,c)$. For the case $\lambda<c$, ``$c$-covering'' is undefinable and thus $m(n,t,\lambda;2,c)=m(n,t,\lambda)$; for the case $\lambda=c$, we have $t\lambda=n$ and thus the pairwise-disjoint property implies the $c$-covering property, that is, we also have $m(n,t,\lambda;2,c)=m(n,t,\lambda)$. Therefore, Theorems~\ref{mainthm-turan density} and~\ref{mainthm:code} naturally generalize the corresponding results on frameproof codes \cite{blackburn2003frameproof,liu2024near} and focal-free hypergraphs and codes \cite{huang2024focal}, which appear as the two extreme cases within our unified framework.


Moreover, when $1\le\lambda\le \min\{s,c-s\}$, the $(k_1,k_2)$-disjoint property becomes trivial and then $m(k,t,\lambda; s+1,c-s+1)=0$; in these cases we can prove $f_{c,s}(n,k)\le \binom{n}{t}/ \binom{k}{t}$ with $t=\lceil sk/c\rceil$.
Meanwhile, a similar discussion from \cite{huang2024focal} inspires us that packings provide a cheap asymptotic lower bound for $f_{c,s}(n,k)$. For integers $n>k>t\ge 2$, an \textit{$(n,k,t)$-packing} is a hypergraph $\mathcal{P}\subseteq \binom{[n]}{k}$ such that $|A\cap B|<t$ for every two distinct $A,B\in \mathcal{P}$. Every $(n,k,t)$-packing $\mathcal{P}$ with $t=\lceil sk/c\rceil$ can be checked to be $(c,s)$-frameproof. Indeed, assume otherwise that there exist $A_0,A_1,\ldots,A_c\subseteq\mathcal{P}$ satisfying $sA_0\subseteq A_1\uplus\cdots \uplus A_c$ (the symbol $\uplus$ denotes the multiset union, and $sA_0$ denote the $s$-fold multiset of $A_0$). Then there exists some $i\in [c]$ such that $|A_0\cap A_i|\ge \lceil s|A_0|/c\rceil= \lceil sk/c\rceil$, a contradiction.
A well-known result of R{\"o}dl \cite{rodl1985packing} showed that for any fixed $k,t$ and  large $n$, there exist asymptotically optimal $(n,k,t)$-packings with $(1-o(1))\binom{n}{t}/\binom{k}{t}$ edges, where $o(1)\to 0$ as $n\to\infty$. Together, we have  when $1\le\lambda\le \min\{s,c-s\}, c\mid (sk-\lambda)$, and $t=\lceil sk/c\rceil$,
\[(1-o(1))\binom{n}{t}/\binom{k}{t}\leq f_{c,s}(n,k)\le \binom{n}{t}/ \binom{k}{t}.\]

An $(n,k,t)$-packing $\mathcal{P}\subseteq \binom{[n]}{k}$ of size $\binom{n}{t}/\binom{k}{t}$ is said to be an $(n,k,t)$-\textit{design}. By the known results on the existence of designs (see \cite{colbourn2010crc,delcourt2024refined,glock2023existence,keevash2014existence}), we are able to further determine the exact values of $f_{c,s}(n,k)$ for infinitely many parameters.

\begin{proposition}\label{prop-general-hypergraph-exact value}
	Let $n\ge k\ge 2, c\ge 2$, $1\le s\le c-1,1\le\lambda\le \min\{s,c-s\}, c\mid (sk-\lambda)$, and $t=\lceil sk/c\rceil$. Let $n\ge n_0=\left(\binom{k}{t}- m(k,t,\lambda; s+1,c-s+1)\right)t+ t-1$. If there exists an $(n,k,t)$-design, then $f_{c,s}(n,k)=\binom{n}{t}/\binom{k}{t}$. In particular, for every sufficiently large $n$ that satisfies $\binom{k-i}{t-i}|\binom{n-i}{t-i}$ for every $0\le i\le t-1$, we have
    \[f_{c,s}(n,k)=\binom{n}{t}\bigg/\binom{k}{t}.\]
\end{proposition}

We also obtain some exact results on $f_{c,s}^{q}(n)$ below.

\begin{theorem}\label{thm-general-code-exact value}
    Let $q\ge c\ge 2$, $1\le s\le c-1$ and let $q=p_1^{e_1}\cdots p_m^{e_m}$ be the canonical integer factorization, where $p_1,\ldots,p_m$ are distinct primes, $e_1,\ldots,e_m$ are positive integers, and $p_1^{e_1}<\cdots< p_m^{e_m}$. If $1\le\lambda\le \min\{s,c-s\}, c\mid (sn-\lambda)$ and $\max\{\frac{2c}{c-s},c-\lambda\}\le n\le p_1^{e_1}+1$, then
    \[f_{c,s}^q(n)= q^{\lceil sn/c\rceil}.\]
\end{theorem}

\subsection{Outline of paper}

The rest of this paper is organized as follows. Sections~\ref{section-hypergraphs} and \ref{section-codes} are devoted to bounds on quantitative frameproof hypergraphs and codes, respectively.
For quantitative frameproof hypergraphs, the upper and lower bounds  in Theorem~\ref{mainthm-turan density} are proved in Subsections~\ref{subsection-hypergraphs-upper} and  \ref{subsection-hypergraphs-lower}, respectively.
For quantitative frameproof codes, we prove asymptotically tight upper and lower bounds in Subsections~\ref{subsection_code_bound}, thereby proving Theorem~\ref{mainthm:code}.
In Section~\ref{section-estimate}, we estimate the generalized matching number.
Finally, we give concluding remarks in Section~\ref{section-conclusion}.

\section{Quantitative frameproof hypergraphs}\label{section-hypergraphs}
Before giving our proofs, we define a small hypergraph/code which violates the $(c,s)$-frameproof property. A family of $c+1$ sets $A_0,A_1,\ldots,A_c\subset 2^{[n]}$, where $A_0\neq A_j$ for each $j\in [c]$ but $A_1,\ldots,A_c$ need not be distinct, is said to be a {\it $(c,s)$-focal hypergraph} with focus $A_0$ if for every $x\in A_0$, we have $|\{j\in [c]:x\in A_j\}|\ge s$. Similarly, a family of $c+1$ vectors $\bm x^0,\bm x^1,\ldots,\bm x^c\subset [q]^n$, where $\bm x^0\neq \bm x^j$ for each $j\in [c]$ but $\bm x^1,\ldots,\bm x^c$ need not be distinct, is said to be a {\it $(c,s)$-focal code} with focus $\bm x^0$ if for every coordinate $i\in [n]$, we have $|\{j\in [c]:x_i^j= x_i^0\}|\ge s$.
Clearly, a hypergraph $\mathcal{F}\subset \binom{[n]}{k}$ is $(c,s)$-frameproof if it does not contain any $(c,s)$-focal hypergraph; a code $\mathcal{C}\subset [q]^n$ is $(c,s)$-frameproof if it does not contain any $(c,s)$-focal code.

In this section, we will present the proof of Theorem \ref{mainthm-turan density}. Our proof follows the idea in \cite{huang2024focal}, where a key definition ``own subset'' is needed.  
Given a hypergraph $\mathcal{F}\subseteq 2^{[n]}$  and an edge $A\in \mathcal{F}$, a subset $T\subseteq A$ is called an \textit{own subset} of $A$ (with respect to $\mathcal{F}$) if for every $B\in \mathcal{F}\backslash \{A\}$, we have $T\nsubseteq B$; otherwise, $T$ is called a \textit{non-own subset} of $A$. We have the following useful observations.


\begin{observation}\label{obs-own subset}
    Let $\mathcal{F}$ be a hypergraph. Then the following hold:
    \begin{enumerate}[(i)]
        \item If $A\in \mathcal{F}$ admits $sA\subseteq T_1\uplus\cdots\uplus T_c$ such that for each $i\in [c]$, $T_i$ is a non-own subset of $A$, then $\mathcal{F}$ contains a $(c,s)$-focal hypergraph with focus $A$.

        \item If $A,A_1,\ldots, A_c\in \mathcal{F}$ form a $(c,s)$-focal hypergraph with focus $A$ such that $sA=A_1\uplus\cdots\uplus A_c$, then the $c$ members $A_1,\ldots,A_c$ are $(s+1,c-s+1)$-disjoint subsets of $A$.
    \end{enumerate}
\end{observation}

We remark that Observation \ref{obs-own subset} (i) and (ii) will be used in the proofs of the upper and lower bounds in Theorem \ref{mainthm-turan density}, respectively. The following fact is also useful in the proofs. For fixed $c,s,k$ with $1\le s\le c-1$ and $t=\lceil sk/c\rceil$,
the number $\lambda\in [c]$ satisfies $\lambda\equiv sk \pmod{c}$ if and only if
    \begin{align}\label{eq-lambda}
        \lambda t+ (c-\lambda)(t-1)= sk.
    \end{align}

\subsection{The upper bounds of $f_{c,s}(n,k)$}\label{subsection-hypergraphs-upper}

In this subsection, we  prove the following upper bound in Theorem \ref{mainthm-turan density}.

\begin{theorem}\label{thm-gen-hypergraph-upperbound}
    For  $n\ge k\ge 2$ and $c\ge 2,1\le s\le c-1$, let $t=\lceil sk/c\rceil$. When $n\ge n_0= \left(\binom{k}{t}- m(k,t,\lambda; s+1,c-s+1)\right)t+ t-1$, we have
    \begin{align}\label{eq-gen-hypergraph-upperbound}
        f_{c,s}(n,k)
        \le
        \frac{
        \binom{n}{t}}
        {\binom{k}{t}- m(k,t,\lambda; s+1,c-s+1)},
    \end{align}
    where $\lambda\in [c]$ satisfies $\lambda\equiv sk \pmod{c}$.
\end{theorem}

First, we give the following useful lemma, which is similar to \cite[Lemma 2.3]{huang2024focal} but more details involved.

\begin{lemma}\label{lem-gen-hypergraph-ownsubset}
    Let $n,k$ and $c,s$ be integers with $n\ge k\ge 2$ and $c\ge 2, 1\le s\le c-1$. Let $\mathcal{F}\subseteq\binom{[n]}{k}$ be a $(c,s)$-frameproof hypergraph and $\mathcal{F}_0= \{A\in \mathcal{F}: A \text{ has no own $(t-1)$-subsets with respect to } \mathcal{F}\}$, where $t=\lceil sk/c\rceil$. Then every $A\in \mathcal{F}_0$ contains at least $\binom{k}{t}-m(k,t,\lambda; s+1,c-s+1)$ own $t$-subsets with respect to $\mathcal{F}$.
\end{lemma}
\begin{proof}
    It suffices to show that each $A\in \mathcal{F}_0$ has at most $m(k,t,\lambda; s+1,c-s+1)$ non-own $t$-subsets. Assume that there exists some $A\in \mathcal{F}_0$ that contains at least $m(k,t,\lambda; s+1,c-s+1)+1$ non-own $t$-subsets. Let $\mathcal{T}_A\subset \binom{A}{t}$ be the set of all non-own $t$-subsets of $A$, then $|\mathcal{T}_A|\geq m(k,t,\lambda; s+1,c-s+1)+1$. So
$\mathcal{T}_A$ contains an $(s+1,c-s+1)$-disjoint collection of $\lambda$ repeatable subsets $T_1,T_2,\ldots,T_\lambda$ of $A$. In order to apply Observation \ref{obs-own subset} (i), we need the following claim.

\begin{claim}\label{claim-disjoint}
        If $T_1,\ldots,T_\lambda$ are $(s+1,c-s+1)$-disjoint subsets of $A$ of size $t$, then we can find subsets $T_{\lambda+1},\ldots,T_c$ of size $t-1$ such that
        \[sA= T_1\uplus T_2\uplus \cdots \uplus T_c.\]
    \end{claim}
    \begin{proof}[Proof of Claim] Since $T_1,\ldots,T_\lambda$ are $(s+1,c-s+1)$-disjoint subsets of $A$, each element of $A$ appears at least $s-(c-\lambda)$ times and at most $s$ times in these subsets.  By the equation $sk= \lambda t+ (c-\lambda)(t-1)$ in (\ref{eq-lambda}), the multiset $E\triangleq kA\setminus (T_1\uplus T_2\uplus \cdots \uplus T_\lambda)$ is of size $(c-\lambda)(t-1)$. Then it suffices to partition $E$ into $c-\lambda$ parts $T_{\lambda+1},\ldots,T_c$ such that each $T_i$ is a set of size $t-1$, $i\in [\lambda+1,c]$. We show that this can be done in a greedy way below.

   Note that each element of $A$ appears in at most $c-\lambda$ times in $E$. For $i\in [c-\lambda]$, let $E_i\subset A$ be the set of elements appearing at least $i$ times in $E$. Then $E_1\supseteq E_2\supseteq\cdots \supseteq E_{c-\lambda}$ and  $E=E_1\uplus \cdots \uplus E_{c-\lambda}$. Since $\sum_{i=1}^{c-\lambda}|E_i|=(c-\lambda)(t-1)$, we have $|E_1|\ge t-1$ and $|E_{c-\lambda}|\le t-1$. In the first step, we construct the $(t-1)$-set $T_{\lambda+1}$. We move all elements from $E_{c-\lambda}$ into it, and move all elements from $E_{c-\lambda-1}$ that have not appeared yet in $T_{\lambda+1}$, and continue moving the elements from $E_{c-\lambda-2}$ that have not appeared, and so on, until the number of elements in $T_{\lambda+1}$ reaches $t-1$. After this step, $E_{c-\lambda}=\emptyset$, $E_1\supseteq E_2\supseteq\cdots \supseteq E_{c-\lambda-1}$ and $\sum_{i=1}^{c-\lambda-1}|E_i|=(c-\lambda-1)(t-1)$, and thus $|E_1|\ge t-1, |E_{c-\lambda-1}|\le t-1$. Then in the second step, we can similarly construct $T_{\lambda+2}$ by  moving all elements from $E_{c-\lambda-1}$, and different elements from $E_{c-\lambda-2}, \ldots$ until $t-1$ elements are moved. Iteratively, we can construct all subsets $T_{\lambda+1},\ldots,T_c$ satisfying the required condition.    \end{proof}

 By the claim, we can find subsets $T_{\lambda+1},T_{\lambda+2},\ldots,T_c\subseteq A$ of size $t-1$ such that
$sA= T_1\uplus T_2\uplus \cdots \uplus T_c$.
By the definition of $\mathcal{T}_A$ and that $A\in \mathcal{F}_0$, all $T_i, i\in [c]$ are non-own subsets of $A$. Therefore, from Observation \ref{obs-own subset} (i),  $\mathcal{F}$ contains a $(c,s)$-focal hypergraph with focus $A$, a contradiction.
\end{proof}

The above lemma shows that in a quantitative frameproof hypergraph, some edges contain many own $t$-subsets. As own $t$-subsets are exclusive and of number at most $\binom{n}{t}$, this contributes to upper bounding $f_{c,s}(n,k)$ as follows.

\begin{proof}[Proof of  Theorem \ref{thm-gen-hypergraph-upperbound}]
Suppose that $\mathcal{F}\in \binom{[n]}{k}$ is a $(c,s)$-frameproof hypergraph. Let $\mathcal{F}_0= \{A\in \mathcal{F}: A \text{ has no own $(t-1)$-subsets with respect to } \mathcal{F}\}$ as in Lemma \ref{lem-gen-hypergraph-ownsubset} and let $\mathcal{F}_1= \mathcal{F}\backslash\mathcal{F}_0$. Let
\[\mathcal{T}_0:= \{T\in \binom{[n]}{t}: T \text{ is an own $t$-subset of some edge in $\mathcal{F}_0$ with respect to } \mathcal{F}\},\]
and let
\[\mathcal{T}_1:= \{T\in \binom{[n]}{t}: T \text{ contains one own $(t-1)$-subset of some edge in } \mathcal{F}_1 \}.\]
By definition of own subsets, $\mathcal{T}_0$ and $\mathcal{T}_1$ are disjoint, and thus $|\mathcal{T}_0|+|\mathcal{T}_1|\le \binom{n}{t}$.
By Lemma \ref{lem-gen-hypergraph-ownsubset},
\begin{align}\label{eq-T0}
    |\mathcal{T}_0|\ge| \mathcal{F}_0|\left(\binom{k}{t}-m(k,t,\lambda; s+1,c-s+1)\right).
\end{align}
For $|\mathcal{T}_1|$, each edge in $\mathcal{F}_1$ contains at least one own $(t-1)$-subset, and the latter generates $n-t+1$ sets of $\mathcal{T}_1$. Meanwhile, each set of $\mathcal{T}_1$ is generated by at most $t$ own $(t-1)$-subsets, which implies that
\begin{align}\label{eq-T1}
    |\mathcal{T}_1|\ge |\mathcal{F}_1|\frac{n-t+1}{t}.
\end{align}

Combining (\ref{eq-T0}) and (\ref{eq-T1}) yields that
\begin{align*}
    \binom{n}{t}&\ge |\mathcal{T}_0|+|\mathcal{T}_1|
    \ge |\mathcal{F}_0|\left(\binom{k}{t}-m(k,t,\lambda; s+1,c-s+1)\right)+ |\mathcal{F}_1|\frac{n-t+1}{t}\\
    &\ge (|\mathcal{F}_0|+|\mathcal{F}_1|)\left(\binom{k}{t}-m(k,t,\lambda; s+1,c-s+1)\right)\\
    &= |\mathcal{F}|\left(\binom{k}{t}-m(k,t,\lambda; s+1,c-s+1)\right)
\end{align*}
as needed, where the last inequality holds when $n\ge \left(\binom{k}{t}- m(k,t,\lambda; s+1,c-s+1)\right)t+ t-1$.
\end{proof}


When $1\le\lambda\le \min\{s,c-s\}$, we have $m(k,t,\lambda; s+1,c-s+1)=0$ and the upper bound (\ref{eq-gen-hypergraph-upperbound}) becomes $\binom{n}{t}/\binom{k}{t}$. As mentioned in Introduction,  this upper bound can be achieved whenever an $(n,k,t)$-design exists with $t=\lceil sk/c\rceil$. Thus Proposition \ref{prop-general-hypergraph-exact value} follows by existence results of designs whenever $n$ is large enough and satisfies $\binom{k-i}{t-i}|\binom{n-i}{t-i}$ for every $0\le i\le t-1$  \cite{delcourt2024refined,glock2023existence,keevash2014existence}.

%

\subsection{The lower bounds of $f_{c,s}(n,k)$}\label{subsection-hypergraphs-lower}

In this subsection, we present a lower bound for $f_{c,s}(n,k)$ which is asymptotically optimal.

\begin{theorem}\label{thm-gen-hypergraph-lowerbound}
    For  $k,c\ge 2$ and $1\le s\le c-1$, let $t=\lceil sk/c\rceil$. Then we have
    \begin{align}
        f_{c,s}(n,k)\ge (1-o(1))\frac{\binom{n}{t}} {\binom{k}{t}- m(k,t,\lambda; s+1,c-s+1)},
    \end{align}
    where $\lambda\in [c]$ satisfies $\lambda\equiv sk \pmod{c}$ and $o(1)\to 0$ as $n\to \infty$.
\end{theorem}

Similar to the proof of \cite[Theorem 2.5]{huang2024focal},  we need the tools of packings and induced packings of hypergraphs, which is originally introduced in \cite{frankl1987colored}.
\begin{definition}[Packings and induced packings]\label{def-hypergraph-packing}
    For a fixed $t$-uniform hypergraph $\mathcal{F}$ and a host $t$-uniform hypergraph $\mathcal{H}$, a family of $m$ $t$-uniform hypergraphs
\[\{(V(\mathcal{F}_1),\mathcal{F}_1), (V(\mathcal{F}_2),\mathcal{F}_2),\ldots,(V(\mathcal{F}_m),\mathcal{F}_m)\}\]
    forms an \textbf{$\mathcal{F}$-packing} in $\mathcal{H}$ if for each $j\in [m]$,
    \begin{enumerate}[(i)]
        \item $V(\mathcal{F}_j)\subseteq V(\mathcal{H}), \mathcal{F}_j\subseteq \mathcal{H}$;
        \item $\mathcal{F}_j$ is a copy of $\mathcal{F}$ defined on the vertex set $V(\mathcal{F}_j)$;
        \item the $m$ $\mathcal{F}$-copies are pairwise edge disjoint, i.e., $\mathcal{F}_i\cap \mathcal{F}_j=\emptyset$ for arbitrary distinct $i,j\in [m]$.
    \end{enumerate}
    The above $\mathcal{F}$-packing is said to be \textbf{induced} if it further satisfies
        \begin{enumerate}[(i)]\setcounter{enumi}{3}
        \item $|V(\mathcal{F}_i)\cap V(\mathcal{F}_j)|\le t$ for arbitrary distinct $i,j\in [m]$;
        \item for $i\neq j$, if $|V(\mathcal{F}_i)\cap V(\mathcal{F}_j)|= t$, then $V(\mathcal{F}_i)\cap V(\mathcal{F}_j)\notin \mathcal{F}_i\cup \mathcal{F}_j$.
    \end{enumerate}
\end{definition}

By definition, every $\mathcal{F}$-packing in $\mathcal{H}$ has at most $|\mathcal{H}|/|\mathcal{F}|$ copies of $\mathcal{F}$, as $\mathcal{F}$-copies are pairwise edge disjoint.
By the influential works of R{\"o}dl \cite{rodl1985packing}, Frankl and R{\"o}dl \cite{frankl1985near}, and Pippenger (see \cite{pippenger1989asymptotic}), near-optimal $\mathcal{F}$-packings with $(1-o(1))|\mathcal{H}|/|\mathcal{F}|$ copies of $\mathcal{F}$ exist. This result was further strengthened by Frankl and F{\"u}redi \cite{frankl1987colored}  that near-optimal induced $\mathcal{F}$-packings exist.
%

\begin{lemma}[\cite{frankl1987colored}]\label{lem-induced packing}
    Let $k>t$ and $\mathcal{F}\subseteq\binom{[k]}{t}$ be fixed. Then there exists an induced $\mathcal{F}$-packing $\{(V(\mathcal{F}_i),\mathcal{F}_i): i\in [m]\}$ in $\binom{[n]}{t}$ with $m\ge (1-o(1))\binom{n}{t}/|\mathcal{F}|$, where $o(1)\to 0$ as $n\to \infty$.
\end{lemma}

Now we give a  proof of Theorem \ref{thm-gen-hypergraph-lowerbound}.

\begin{proof}[Proof of Theorem \ref{thm-gen-hypergraph-lowerbound}]
    Let $\mathcal{G}\subseteq\binom{[k]}{t}$ be one of the largest $t$-uniform hypergraphs on $k$ vertices that contains no $\lambda$ many $(s+1,c-s+1)$-disjoint edges. Then $|\mathcal{G}|=m(k,t,\lambda; s+1,c-s+1)$. Let $\mathcal{F}:=\binom{[k]}{t}\backslash \mathcal{G}$ be a hypergraph on $[k]$, then $|\mathcal{F}|= \binom{k}{t}-|\mathcal{G}|= \binom{k}{t}-m(k,t,\lambda; s+1,c-s+1)$.
    We apply Lemma \ref{lem-induced packing} with $\mathcal{F}$ to get  an induced $\mathcal{F}$-packing $\{(V(\mathcal{F}_i),\mathcal{F}_i): i\in [m]\}$ in $\binom{[n]}{t}$ with
    \[m\ge (1-o(1))\binom{n}{t}/|\mathcal{F}|= (1-o(1))\frac{\binom{n}{t}} {\binom{k}{t}-m(k,t,\lambda; s+1,c-s+1)},\]
    where $o(1)\to 0$ as $n\to \infty$. 
    Let \[\mathcal{H}(\mathcal{F}):= \{V(\mathcal{F}_i): i\in [m]\}\subseteq \binom{[n]}{k}.\]
Then $|\mathcal{H}(\mathcal{F})|=m$. It suffices to show that $\mathcal{H}(\mathcal{F})$
   is $(c,s)$-frameproof.

    Suppose on the contrary that $V_0,V_1,\ldots,V_c\in \mathcal{H}(\mathcal{F})$ form a $(c,s)$-focal hypergraph with focus $V_0$, where $V_0\neq V_j$ for each $j\in [c]$ but $V_1,\ldots,V_c$ are no need to be distinct. Then, $sV_0\subseteq V_1\uplus\cdots\uplus V_c$, that is, we can find $A_1\subseteq V_0\cap V_1,\ldots, A_c\subseteq V_0\cap V_c$ such that $sV_0= A_1\uplus\cdots\uplus A_c$. It follows from Observation \ref{obs-own subset} (ii) that the $c$ members $A_1,\ldots,A_c$ are $(s+1,c-s+1)$-disjoint subsets of $V_0$. By Definition \ref{def-hypergraph-packing} (iv),
   for each $i\in [c]$, $|A_i|\le |V_0\cap V_i|\le t$. Moreover, by (\ref{eq-lambda}),  there are at least $\lambda$ distinct $i\in [c]$ such that $|A_i|=t$ and thus $A_i=V_0\cap V_i$. Assume without loss of generality that $|V_0\cap V_i|=t$ for each $i\in [\lambda]$, and $A_1=V_0\cap V_1,\ldots,A_\lambda=V_0\cap V_\lambda$ are $\lambda$ many $(s+1,c-s+1)$-disjoint edges over $V_0$.
   Suppose $V_0=V(\mathcal{F}_{i_0})$ for some $i_0\in [m]$.
   By Definition \ref{def-hypergraph-packing} (v),  for each $i\in [\lambda]$, $V_0\cap V_i\notin \mathcal{F}_{i_0}$. However, by the definition of $\mathcal{G}$, the hypergraph $\binom{V_0}{t}/\mathcal{F}_{i_0}$ contains no $(s+1,c-s+1)$-disjoint collection of size $\lambda$, a contradiction. This completes the proof.
\end{proof}

\section{Quantitative frameproof codes}\label{section-codes}

\subsection{Upper and lower bounds of $f_{c,s}^q(n)$}\label{subsection_code_bound}
Similar to the own subsets of a set, we need to define the own subsequences of a vector as in \cite{huang2024focal}.
For a vector $\bm x=(x_1,\ldots,x_n)\in [q]^n$ and a subset $T\subseteq[n]$, let $\bm x_T=(x_i:i\in T)$ denote the subsequence of $\bm x$ restricted on $T$. For a code $\mathcal{C}\subseteq[q]^n$, a codeword $\bm x\in \mathcal{C}$ and a subset $T\subseteq [n]$, the subsequence $\bm x_T$ is called an \textit{own subsequence} of $\bm x$ (with respect to $\mathcal{C}$) if for every $\bm y\in \mathcal{C}\backslash\{\bm x\}$, $\bm x_T\neq \bm y_T$; otherwise, $\bm x_T$ is called a \textit{non-own subsequence} of $\bm x$.

The proof of Theorem \ref{mainthm:code} is similar to that of Theorem \ref{mainthm-turan density}, except an appropriate version of Lemma \ref{lem-induced packing} recently developed in \cite{liu2024near}. Therefore, we state the upper and lower bounds below but postpone their proofs to Appendix \ref{sec:main-code}.

\begin{theorem}\label{thm-gen-code-upperbound}
    For integers $n,c\ge 2$ and $1\le s\le c-1$, let $t=\lceil sn/c\rceil$. When $q\ge \frac{t}{n-t+1}(\binom{n}{t}-m(n,t,\lambda;s+1,c-s+1))$, we have
    \[f_{c,s}^{q}(n)\le \frac{\binom{n}{t}}{\binom{n}{t}-m(n,t,\lambda;s+1,c-s+1)} q^t,\]
    where $\lambda\in [c]$ is the integer such that $\lambda\equiv sn \pmod{c}$.
\end{theorem}

%

\begin{theorem}\label{thm-gen-code-lowerbound}
    For any integers $n,c\ge 2$ and $1\le s\le c-1$, let $t=\lceil sn/c\rceil$. Then we have
    \[f_{c,s}^{q}(n)\ge (1-o(1))\frac{\binom{n}{t}}{\binom{n}{t}-m(n,t,\lambda;s+1,c-s+1)} q^t,\]
    where $\lambda\in [c]$ is the integer such that $\lambda\equiv sn \pmod{c}$, and $o(1)\to 0$ as $q\to \infty$.
\end{theorem}

\subsection{Exact values of $f_{c,s}^q(n)$}

In this subsection, we  prove Theorem \ref{thm-general-code-exact value}.
First, note that $m(n,\lceil sn/c\rceil,\lambda;s+1,c-s+1)=0$ for $1\le \lambda\le \min\{s,c-s\}$. By Theorem \ref{thm-gen-code-upperbound}, for $sn\equiv \lambda \pmod{c}$ with $1\le\lambda\le \min\{s,c-s\}$ and sufficiently large $q\ge 2^{\Omega(n)}$, we have $f_{c,s}^q(n)\le q^{\lceil sn/c\rceil}$. The following theorem shows that the same upper bound also holds for smaller $q$.
\begin{theorem}\label{thm-exactvalue-upperbound}
    Let $c\ge 2$ and $1\le s\le c-1$. Suppose that $q> c-\lambda$ with $\lambda\equiv sn\pmod{c}$ and $1\le\lambda\le \min\{s,c-s\}$. Then
    \[f_{c,s}^q(n)\le q^{\lceil sn/c\rceil}.\]
\end{theorem}
\begin{proof}
    Let $t=\lceil sn/c\rceil$, then $sn-\lambda=c(t-1)$. Let $\mathcal{C}\subseteq [q]^n$ be a $(c,s)$-frameproof code. It suffices to show that $|\mathcal{C}|\le q^t$. 
    For a subset $S\subseteq [n]$, let
    \[U_S= \{\bm x\in\mathcal{C}:\bm x_S \text{ is an own subsequence of $\bm x$ with respect to } \mathcal{C}\}.\]
    By the $(c,s)$-frameproof property, for every $s$-fold partition $s[n]=T_1\uplus\cdots \uplus T_c$, we have
    \begin{align}\label{eq-Cpartition}
        \mathcal{C}= U_{T_1}\cup\cdots\cup U_{T_c}.
    \end{align}

    Let $S\in \binom{[n]}{t-1}$ be a subset such that $|U_S|$ is maximized. Suppose that $q> c-\lambda$ and $|\mathcal{C}|\ge q^t+1$. We will show that there exists some $S'\in\binom{[n]}{t-1}$ such that $|U_{S'}|\ge |U_S|+1$, a contradiction.

    Assume without loss of generality that $S=[t-1]$. 
    For every $\bm s\in [q]^{t-1}$ which is not an own subsequence of any $\bm x\in \mathcal{C}$ on $[t-1]$, there are at most $q$ choices of $\bm x\in \mathcal{C}$ such that for some $a\in[q]$, $(\bm s,a)\in [q]^{t}$ is an own subsequence of some $\bm x$ on $[t]$. This implies that
    $|U_{[t]}\backslash U_{[t-1]}|\le q(q^{t-1}-|U_{[t-1]}|)$.
    Therefore, we have
    \begin{align}\label{eq-Ud+1}
        |U_{[t]}|\le |U_{[t-1]}|+|U_{[t]}\backslash U_{[t-1]}|\leq q^t-(q-1)|U_{[t-1]}|.
    \end{align}

    Consider $\lambda$ identical sets $T_i=[t]$ for $i\in [\lambda]$. Since $\lambda\le \min\{s,c-s\}$, these $\lambda$ sets are trivially $(s+1,c-s+1)$-disjoint. By Claim \ref{claim-disjoint}, we can find subsets $T_{\lambda+1},\ldots,T_c$ of size $t-1$ such that
    $s[n]= T_1\uplus T_2\uplus \cdots \uplus T_c$.
    Then by (\ref{eq-Cpartition}) and (\ref{eq-Ud+1}), we have
    \[|U_{T_{\lambda+1}}\cup\cdots\cup U_{T_c}|\ge |\mathcal{C}|-|U_{[t]}|\ge (|\mathcal{C}|-q^{t})+(q-1)|U_{[t-1]}|\ge 1+(q-1)|U_{[t-1]}|.\]
    By pigeonhole principle, there exists some $i$ with $\lambda+1\le i\le c$ such that
    \[|U_{T_i}|\ge \left\lceil\frac{1+(q-1)|U_{[t-1]}|}{c-\lambda} \right\rceil \ge |U_{[t-1]}|+1,\]
    where the last inequality holds since $q> c-\lambda$. Setting $S'=T_i$, we have $S'\in \binom{[n]}{t-1}$ and $|U_{S'}|\ge |U_{[t-1]}|+1$, contradicted to the maximality of $U_{[t-1]}$. This completes the proof.
\end{proof}

Clearly, Theorem \ref{thm-exactvalue-upperbound} implies the upper bound of Theorem \ref{thm-general-code-exact value}. 
For the lower bound, we prove it by connecting error-correcting codes with large minimum distance to quantitative frameproof codes. For any two vectors $\bm x,\bm y\in [q]^n$, let $I(\bm x,\bm y)=\{i\in [n]:x_i= y_i\}$, then the \textit{Hamming distance} between $\bm x,\bm y$ is $d(\bm x,\bm y)=n-|I(\bm x,\bm y)|$. The \textit{minimum distance} of a code $\mathcal{C}\subseteq [q]^n$ is $d(\mathcal{C})=\min\{d(\bm x,\bm y): \text{distinct } \bm x,\bm y\in\mathcal{C}\}$.

\begin{lemma}\label{lem-mindistance}
    If $\mathcal{C}\in[q]^n$ satisfies $d(\mathcal{C})> \lfloor\frac{c-s}{c}n \rfloor$, then $\mathcal{C}$ is $(c,s)$-frameproof.
\end{lemma}
\begin{proof}
    Assume otherwise that $\bm x^0,\bm x^1,\ldots,\bm x^c\in \mathcal{C}$ form a $(c,s)$-focal code with focus $\bm x^0$, where $\bm x^0\neq \bm x^j$ for each $j\in [c]$ but $\bm x^1,\ldots,\bm x^c$ are no need to be distinct. Then, $s[n]\subseteq I(\bm x^0,\bm x^1)\uplus\cdots\uplus I(\bm x^0,\bm x^c)$. This implies that $\sum_{j=1}^{c}d(\bm x^0,\bm x^j)= \sum_{j=1}^{c}(n-|I(\bm x^0,\bm x^i)|)\le (c-s)n$. So there exists some $j\in [c]$ such that $d(\bm x^0,\bm x^j)\le\lfloor\frac{c-s}{c}n \rfloor$, a contradiction.
\end{proof}

Applying Lemma \ref{lem-mindistance} to codes generated by orthogonal arrays in \cite[Theorem III.7.18 and Theorem III.7.20]{colbourn2010crc}, we get the lower bound in Theorem \ref{thm-general-code-exact value} as in \cite{huang2024focal}.

\section{Critical quantitative frameproof codes and hypergraphs}

In this subsection, we consider a critical version of quantitative frameproof codes and hypergraphs.

Recall that in Introduction, we define the quantitative frameproof codes by using a $(c,s)$-focal code $\bm x^0,\bm x^1,\ldots,\bm x^c\subset [q]^n$ with focus $\bm x^0$, where $\bm x^1,\ldots,\bm x^c$ are no need to be distinct. Actually, when $s=1$ or $s=c-1$, it is equivalent to require that $\bm x^1,\ldots,\bm x^c$ are pairwise distinct, and we get the definitions of frameproof codes or focal-free codes. So it is natural to consider the structures by requiring that $\bm x^1,\ldots,\bm x^c$ are pairwise distinct,
which we call {\it critical quantitative frameproof codes}.
Formally speaking, a code $\mathcal{C}\subset [q]^n$ is said to be {\it $(c,s)$-critical-frameproof} if for every $c+1$ distinct codewords $\bm x^0,\mathbf{x}^1,\ldots,\bm x^c$, there must be some coordinate $i\in [n]$ such that $|\{j\in [c]:x_i^j=x_i^0\}|<s$.

When $q=2$, we can similarly define  {\it critical quantitative frameproof hypergraphs}.
A $k$-uniform hypergraph $\mathcal{F}\subseteq 2^{[n]}$  is said to be {\it $(c,s)$-critical-frameproof} if for every $c+1$ distinct edges $A_0,A_1,\ldots,A_c$, there must be some vertex $i\in A_0$ such that $|\{j\in [c]:i\in A_j\}|<s$.

Let $g_{c,s}(n,k)$ denote the maximum size of a $(c,s)$-critical-frameproof hypergraph in $\binom{[n]}{k}$. Let $g_{c,s}^{q}(n)$ denote the maximum size of a $(c,s)$-critical-frameproof code in $[q]^n$. We present asymptotically optimal estimates of both $g_{c,s}(n,k)$ for large $n$ and $g_{c,s}^{q}(n)$ for large $q$ as follows, up to a constant factor $\min\{s,c-s\}$.

\begin{theorem}\label{thm-gen-hypergraph-critical}
    For $n\ge k\ge 2$ and $c\ge 2,1\le s\le c-1$, let $t=\lceil sk/c\rceil$. Let $\lambda\in [c]$ satisfy $\lambda\equiv sk \pmod{c}$. Then
    \begin{enumerate}[(i)]
      \item \[ g_{c,s}(n,k)\ge f_{c,s}(n,k)\ge (1-o(1))\frac{\binom{n}{t}} {\binom{k}{t}- m(k,t,\lambda; s+1,c-s+1)},\]
            where $o(1)\to 0$ as $n\to \infty$.
      \item When $n\ge \left(\binom{k}{t}- m(k,t,\lambda; s+1,c-s+1)\right)\frac{t+c}{s_0}+ t-1$, we have
            \[g_{c,s}(n,k)\le \frac{s_0\binom{n}{t}} {\binom{k}{t}- m(k,t,\lambda; s+1,c-s+1)},\]
            where $s_0=\min\{s,c-s\}$.
    \end{enumerate}
\end{theorem}

\begin{theorem}\label{thm-gen-code-critical}
    For any integers $n,c\ge 2$ and $1\le s\le c-1$, let $t=\lceil sn/c\rceil$. Let $\lambda\in [c]$ satisfy $\lambda\equiv sn \pmod{c}$. Then
    \begin{enumerate}[(i)]
      \item \[g_{c,s}^{q}(n)\ge f_{c,s}^{q}(n) \geq (1-o(1))\frac{\binom{n}{t}}{\binom{n}{t}-m(n,t,\lambda;s+1,c-s+1)} q^t,\]
            where $o(1)\to 0$ as $q\to \infty$.
      \item When $q\ge \frac{t}{s_0(n-t+1)}\left(\binom{n}{t}-m(n,t,\lambda;s+1,c-s+1)\right)$ and $q\ge (\frac{c+1}{s_0})^{1/t}$, we have
            \[g_{c,s}^{q}(n)\le \frac{s_0\binom{n}{t}}{\binom{n}{t}-m(n,t,\lambda;s+1,c-s+1)} q^t,\]
            where $s_0=\min\{s,c-s\}$.
    \end{enumerate}
\end{theorem}

The two lower bounds in Theorems~\ref{thm-gen-hypergraph-critical} and Theorem~\ref{thm-gen-code-critical} are trivial.
 For the upper bounds, we can follow all the  proofs of quantitative frameproof codes and hypergraphs but with a fine adjustment of the definition ``own subset''.

We define the quantitative own subset as follows. Let $\mathcal{F}\subseteq 2^{[n]}$ be a hypergraph and $A\in \mathcal{F}$ be an edge. A subset $T\subseteq A$ is called an \textit{$s_0$-own subset} of $A$ (with respect to $\mathcal{F}$) if there exist no $s_0$ distinct $B_1,B_2,\ldots,B_{s_0}\in \mathcal{F}\backslash \{A\}$ such that $T\subseteq B_i$ for all $j\in[s_0]$; otherwise, $T$ is called an \textit{$s_0$-non-own subset} of $A$ (with respect to $\mathcal{F}$).
We also define a small hypergraph which violates the $(c,s)$-critical-frameproof property. A family of $c+1$ distinct sets $A_0,A_1,\ldots,A_c\subset 2^{[n]}$ is said to be a {\it $(c,s)$-critical-focal hypergraph} with focus $A_0$ if for every $x\in A_0$, we have $|\{j\in [c]:x\in A_j\}|\ge s$.

Now, we are ready to present the counterpart of Observation~\ref{obs-own subset}. We point out that the new observation below becomes nontrivial, and thus we give its proof.

\begin{observation}
    Let $\mathcal{F}$ be a hypergraph of cardinality at least $c+1$ and $s_0=\min\{s,c-s\}$. Then the following hold:
    \begin{enumerate}[(i)]
        \item If $A\in \mathcal{F}$ admits $sA\subseteq T_1\uplus\cdots\uplus T_c$ such that for each $i\in [c]$, $T_i$ is an $s_0$-non-own subset of $A$, then $\mathcal{F}$ contains a $(c,s)$-critical-focal hypergraph with focus $A$.

        \item If $A,A_1,\ldots, A_c\in \mathcal{F}$ form a $(c,s)$-critical-focal hypergraph with focus $A$ such that $sA=A_1\uplus\cdots\uplus A_c$, then the $c$ members $A_1,\ldots,A_c$ are $(s+1,c-s+1)$-disjoint subsets of $A$.
    \end{enumerate}
\end{observation}
\begin{proof}
Note that  (ii)  is trivial. For (i), we aim to find distinct $B_1,\ldots,B_c\in \mathcal{F}\backslash\{A\}$ such that $sA\subseteq B_1\uplus\cdots\uplus B_c$.

\textbf{Case 1:} $s_0=c-s$. By the definition of $s_0$-non-own subsets, for each $i\in [c]$, we can find $s_0$ distinct edges $B_{i,1},B_{i,2},\ldots,B_{i,s_0}\in \mathcal{F}\backslash\{A\}$ such that $T_i\subset B_{i,j},j\in[s_0]$.
We construct a bipartite graph $G=(X\cup Y,E)$ with $X$ being the multiset $\{\{T_1,\ldots,T_c\}\}$ and $Y=\{B_{i,j}:i\in[c],j\in[s_0]\}$, where each $T_\ell$ is adjacent to $B_{i,j}$ if and only if $T_\ell\subset B_{i,j}$.
For each vertex $T_i\in X$, its degree $d(T_i)\ge s_0= c-s$. For each vertex $B_{i,j}\in Y$, notice that $T_1,\ldots,T_c$ are $(c-s+1)$-covering with respect to $A$, so $d(B_{i,j})\le c-s$. Otherwise, we have $B_{i,j}=A$, contradicting to $B_{i,j}\in \mathcal{F}\backslash\{A\}$.
By Hall's theorem, we obtain a matching of $G$ covering $X$, which implies that we find $c$ distinct edges $B_1,\ldots,B_c\in Y$ such that $T_i\subset B_i,i\in[c]$. Therefore, $sA\subseteq B_1\uplus\cdots\uplus B_c$, i.e., $A,B_1,\ldots,B_c$ form a $(c,s)$-critical-focal hypergraph with focus $A$.


\textbf{Case 2:}  $s_0=s$. In this case, $d(T_i)\geq s_0=s$ and $d(B_{i,j})\le c-s$, but $s\leq c-s$, so we can not apply Hall's theorem directly to find $B_1,\ldots,B_c$. Instead, we follow an algorithm as follows.
At step $1$, we take $B_1=B_{1,1}$ with respect to $T_1$ and set $S_1=\emptyset$.
At step $i$, suppose that we have taken $B_1,\ldots,B_{i-1}$ with respect to $T_{i-1}$ and set some $S_{i-1}$. Then from $B_{i,1},\ldots,B_{i,s}$ with respect to $T_i$:
\begin{enumerate}[(1)]
  \item if $\{B_{i,1},\ldots,B_{i,s}\}\backslash\{B_1,\ldots,B_{i-1}\}\neq \emptyset$, then we take some $B_i\in \{B_{i,1},\ldots,B_{i,s}\}\backslash\{B_1,\ldots,B_{i-1}\}$ and set $S_i=S_{i-1}$;
  \item if $\{B_{i,1},\ldots,B_{i,s}\}\subseteq \{B_1,\ldots,B_{i-1}\}$, let $B_i$ be nothing and set $S_i=S_{i-1}\cup T_i$. Note that $T_i$ is contained in at least $s$ edges of $\{B_1,\ldots,B_{i-1}\}$.
\end{enumerate}
The algorithm stops after $c$ steps. For some $c_0\leq c$,  we obtain distinct edges $B_1,\ldots,B_{c_0}$ and $S_c$ as the union of several $T_i$'s.
By the algorithm, $s(A\backslash S_c)\subseteq B_1\uplus\cdots\uplus B_{c_0}$ and each $j\in S_c$ is contained in at least $s$ edges of $\{B_1,\ldots,B_{c_0}\}$. Hence, we have $sA\subseteq B_1\uplus\cdots\uplus B_{c_0}$. Randomly choosing other $c-c_0$ distinct edges, we complete the proof.
\end{proof}

We omit the rest of the proofs of Theorems~\ref{thm-gen-hypergraph-critical} and~\ref{thm-gen-code-critical} as they are routine based on the previous proofs of quantitative frameproof hypergraphs and codes.

\section{Estimating generalized matching numbers}\label{section-estimate}
In this section, we study upper and lower bounds for the generalized matching number $m(n, t, \lambda;s_1+1,s_2+1)$, deriving less precise but more explicit bounds for quantitative frameproof hypergraphs or codes. By the remark after Definition~\ref{def-disjoint}, when $\lambda\ge s_1+ s_2+ 1$, an $(s_1+1,s_2+1)$-disjoint collection of size $\lambda$ never exists, thus $m(n, t, \lambda;s_1+1,s_2+1)= \binom{n}{t}$.

%
%
%

Note that in the context of quantitative frameproof hypergraphs or codes, we set $s_1=s, s_2=c-s$, and naturally $\lambda\le c= s_1+ s_2$. Below we will estimate the generalized matching number $m(n,t,\lambda;s_1+1,s_2+1)$ under the condition $\lambda\le s_1+ s_2$.

\subsection{An Upper Bound of $m(n,t,\lambda;s_1+1,s_2+1)$}
We give our upper bound below, while the main idea is motivated from \cite{blackburn2003frameproof}.

\begin{theorem}\label{thm-general-matching-upperbound}
Let $n,t,\lambda,s_1,s_2$ be positive integers with $n>t>1$ and $\min\{s_1+1,s_2+1\}\leq \lambda\le s_1+ s_2$.
 Set $\chi= \max\{ \lceil \frac{t}{s_1}\rceil, \lceil \frac{n-t}{s_2}\rceil\}$. Then
\[m(n, t, \lambda;s_1+1,s_2+1)\le \frac{1}{n}\binom{n}{t}(\lambda-1) \left\lceil\frac{n}{\lfloor n/\chi\rfloor}\right\rceil.\]
\end{theorem}

Before proving the upper bound, we consider a simpler situation. Let $\mathbb{Z}_n$ denote the set of integers modulo $n$. For each $a\in \mathbb{Z}_n$, define a continuous subset $T(a)\subseteq \mathbb{Z}_n$ as
\begin{center}
$T(a)= \{a, a+1, \ldots, a+t-1\}.$  \end{center}
Write $\T= \{T(a): a\in \mathbb{Z}_n \}$.

\begin{lemma}\label{lem-general-matching-simpler}
Let $n,t,\lambda,s_1,s_2$ be positive integers with $n>t>1$ and $\min\{s_1+1,s_2+1\}\leq \lambda\le s_1+ s_2$. Define $\chi$, $T(a)$ and $\T$ as above. Suppose that a family $\mS\subseteq \T$ contains no $(s_1+1,s_2+1)$-disjoint collection of size $\lambda$. Then
\[|\mS|\le (\lambda-1) \left\lceil\frac{n}{\lfloor n/\chi\rfloor}\right\rceil.\]
\end{lemma}

\begin{proof}
Let $m=\lfloor n/\chi\rfloor$ and $n=m\gamma-n_0$ with $\gamma\in \mathbb{Z},0\le n_0<m$. Then $\gamma=\lceil n/m\rceil=\lceil \frac{n}{\lfloor n/\chi\rfloor}\rceil$. For $1\le i\le \gamma-1$, define
\begin{align*}
\T_i=\{&T(i), T(i+\gamma),\ldots,T(i+(m-n_0)\gamma),\\
       &T(i+(m-n_0)\gamma+(\gamma-1)),\ldots,T(i+(m-n_0)\gamma+(n_0-1)(\gamma-1))\},
\end{align*}
and
\[\T_\gamma= \{T(\gamma), T(2\gamma), \ldots, T((m-n_0)\gamma)\}.\]
Clearly, all $\T_i$ are pairwise disjoint and $\T=\cup\T_i$.
We claim that each $\T_i$ is $(s_1+1,s_2+1)$-disjoint. Indeed, since all members of $\T_i$ are of length $t$ and generated by iteratively right shifting $T(i)$ with 
$\lceil \frac{n}{\lfloor n/\chi\rfloor}\rceil$ or $\lfloor \frac{n}{\lfloor n/\chi\rfloor}\rfloor$ positions, one can verify that each element in $\mathbb{Z}_n$ is contained in at most
$\lceil \frac{t}{\lfloor \frac{n}{\lfloor n/\chi\rfloor}\rfloor}\rceil\le \lceil \frac{t}{\chi}\rceil\le \lceil \frac{t}{\lceil t/s_1\rceil}\rceil\le s_1$
members of $\T_i$. This implies the $(s_1+1)$-disjoint property. And the $(s_2+1)$-covering property comes immediately from the value of $\chi$ and the symmetry between $s_1$ and $s_2$.

Note that $\mS$ contains no $(s_1+1,s_2+1)$-disjoint collection of size $\lambda$, then by the claim, $|\mS\cap \T_i|\le \lambda-1$ for $1\le i\le \gamma$. Therefore,
\[|\mS|= |\mS\cap\T|= |\mS\cap (\cup \T_i)|= \sum|\mS\cap \T_i|\le (\lambda-1)\gamma=(\lambda-1) \left\lceil\frac{n}{\lfloor n/\chi\rfloor}\right\rceil. \]
\end{proof}

Now we are ready to prove the upper bound in Theorem \ref{thm-general-matching-upperbound}.

\begin{proof}[Proof of Theorem~\ref{thm-general-matching-upperbound}]
Let $\mS\subseteq \binom{[n]}{t}$ be a family without $(s_1+1,s_2+1)$-disjoint collection of size $\lambda$. Define $\T$ as above. We will upper bound $|\mS|$ by double counting
\[Q= \{(\alpha, S): S\in \mS \text{ and } \alpha:[n]\to \mathbb{Z}_n \text{ is a bijection such that } \alpha(S)\in \T\},\]
where $\alpha(S)= \{\alpha(i): i\in S\}$.

There are $|\mS|$ choices for $S\in \mS$. For fixed $S$, there are $n\cdot t!(n-t)!$ choices for $\alpha(S)\in \T$. Hence $|Q|= |\mS|n\cdot t!(n-t)!$. On the other hand, there are $n!$ choices for a bijection $\alpha$. For fixed $\alpha$, by Lemma~\ref{lem-general-matching-simpler}, there are at most $(\lambda-1)\lceil\frac{n}{\lfloor n/\chi\rfloor}\rceil$ choices for $S$ such that $S\in \mS$ and $\alpha(S)\in \T$. Hence $|Q|\le n!(\lambda-1)\lceil\frac{n}{\lfloor n/\chi\rfloor}\rceil$. Therefore,
\[|\mS|= |Q|/\left(n\cdot t!(n-t)!\right)\le n!(\lambda-1)\left\lceil\frac{n}{\lfloor n/\chi\rfloor}\right\rceil/(n\cdot t!(n-t)!)= \frac{1}{n}\binom{n}{t} (\lambda-1)\left\lceil\frac{n}{\lfloor n/\chi\rfloor}\right\rceil.\]
\end{proof}

\begin{corollary}\label{cor-focal-upperbound-explicit}
Let $n,c,s,\lambda$ be integers with $n,c\ge 2, s\in [c-1]$, and $n\ge c(c-1)$. Let 
 $t=\lceil sn/c\rceil$. Suppose that $\min\{s+1,c-s+1\}\leq \lambda\le c$. Then
\[m(n, t, \lambda;s+1,c-s+1)\le \frac{\lambda-1}{n}\left\lceil\frac{n}{c-1}\right\rceil \binom{n}{t}.\]
If further $c\mid n$,  we  have $m(n, t, \lambda;s+1,c-s+1)\le\frac{\lambda-1}{c}\binom{n}{t}$.
\end{corollary}

\begin{proof}
Take $s_1=s, s_2=c-s$ and $\chi$ defined as above, then
\begin{center}
$\chi= \max\{ \lceil \frac{t}{s_1}\rceil, \lceil \frac{n-t}{s_2}\rceil\}
= \max\{ \lceil \frac{\lceil sn/c\rceil}{s}\rceil, \lceil \frac{n-\lceil sn/c\rceil}{c-s}\rceil\}
= \lceil \frac{\lceil sn/c\rceil}{s}\rceil= \lceil \frac{n}{c}\rceil$.
\end{center}
Since $n\ge c(c-1)$, a simple calculation turns out that $\lceil\frac{n}{\lfloor n/\chi\rfloor}\rceil\le \lceil\frac{n}{c-1}\rceil$.
Then by Theorem~\ref{thm-general-matching-upperbound}, we have
\begin{align*}
m(n, t, \lambda;s+1,c-s+1)&\le \frac{\lambda-1}{n}\left\lceil\frac{n}{c-1}\right\rceil \binom{n}{t}.
\end{align*}
When $c\mid n$, we easily have $\lceil\frac{n}{\lfloor n/\chi\rfloor}\rceil= \frac{n}{c}$, and thus
$m(n, t, \lambda;s+1,c-s+1)\le \frac{\lambda-1}{c}\binom{n}{t}$.
\end{proof}
Note that we do not require $\lambda\equiv sn\pmod{c}$ in Corollary~\ref{cor-focal-upperbound-explicit}.

\subsection{A Lower Bound of $m(n,t,\lambda;s_1+1,s_2+1)$}
First, we define $m(n,t,\lambda;s)$ to be the maximum size of a family $\mF \subseteq \binom{[n]}{t}$ such that $\mF$ contains no $s$-disjoint collection of size $\lambda$. We claim that
\begin{equation}\label{eq-general-matching-reduce}
m(n,t,\lambda;s_1+1,s_2+1)\ge \max \{m(n,t,\lambda;s_1+1), m(n,n-t,\lambda;s_2+1)\}.
\end{equation}
Indeed, for any family $\F\subseteq \binom{[n]}{t}$ without $(s_1+1)$-disjoint collection of size $\lambda$, clearly $\F$ contains no $(s_1+1,s_2+1)$-disjoint collection of size $\lambda$. This implies that $m(n,t,\lambda;s_1+1,s_2+1)\ge m(n,t,\lambda;s_1+1)$.
Since $\F^c= \{[n]\backslash F: F\in \F\}\subseteq \binom{[n]}{n-t}$, a symmetric argument implies that $m(n,t,\lambda;s_1+1,s_2+1)\ge m(n,n-t,\lambda;s_2+1)$.

Eq.~(\ref{eq-general-matching-reduce}) helps to reduce lower bounding $m(n,t,\lambda;s_1+1,s_2+1)$ to lower bounding a simpler situation $m(n,t,\lambda;s_1+1)$ (or $m(n,n-t,\lambda;s_2+1)$). We will give a lower bound for $m(n,t,\lambda;s_1+1)$ below.

\begin{lemma}\label{lem-general-matching-lowerbound}
Let $n,t,\lambda,s$ be positive integers. Then
\[m(n, t, \lambda;s+1)\ge \binom{n}{t}- \binom{n-\lceil \lambda/s\rceil+ 1}{t}.\]
\end{lemma}

\begin{proof}
When $\lambda\le s$, we have both $m(n, t, \lambda;s+1)=0$ and $\binom{n}{t}- \binom{n-\lceil \lambda/s\rceil+ 1}{t}= \binom{n}{t}- \binom{n}{t}=0$, hence the equality holds.

Now assume that $\lambda\ge s+1$ and then $\lceil \frac{\lambda}{s}\rceil\ge 2$. Take a nonempty subset $S\subseteq  [n]$ with $|S|= \lceil \frac{\lambda}{s}\rceil- 1$ and a family $\F\subseteq \binom{n}{t}$ consisting of all $t$-sets intersecting $S$. We claim that $\F$ contains no $(s+1)$-disjoint collection of size $\lambda$. Otherwise, suppose that there exist $F_1,F_2,\ldots, F_\lambda$ in $\F$ forming an $(s+1)$-disjoint collection. Since $F_i\cap S\neq \emptyset$, each $F_i$ contains at least one element of $S$. By Pigeonhole Principle and $|S|s= (\lceil \frac{\lambda}{s}\rceil- 1)s< \lambda$, there exist $s+1$ many $F_i$'s intersecting at some element of $S$, contradicted to the $(s+1)$-disjoint property. Therefore,
\[m(n, t, \lambda;s+1)\ge |\F|= \binom{n}{t}- \binom{n-\lceil \lambda/s\rceil+ 1}{t}.\]
\end{proof}

\begin{theorem}\label{thm-general-matching-lowerbound}
Let $n,t,\lambda,s_1,s_2$ be positive integers with $\min\{s_1+1,s_2+1\}\leq \lambda\le s_1+ s_2$. Then
\[m(n,t,\lambda;s_1+1,s_2+1)\ge \max \left\{\binom{n}{t}- \binom{n-\lceil \lambda/s_1\rceil+ 1}{t}, \binom{n}{t}- \binom{n-\lceil \lambda/s_2\rceil+ 1}{n-t}\right\}.\]
\end{theorem}

\begin{proof}
Combining Eq.~(\ref{eq-general-matching-reduce}) and Lemma~\ref{lem-general-matching-lowerbound}, we immediately complete the proof.
\end{proof}

\begin{corollary}\label{cor-focal-lowerbound-explicit}
Let $n,c,s,\lambda$ be integers with $n,c\ge 2, s\in [c-1]$,  and $n\ge c(c-1)$. Suppose that $\min\{s+1,c-s+1\}\leq \lambda\le c$. Let $t=\lceil sn/c\rceil$ and $s_0=\min \{s,c-s\}$. Then
\[
m(n, t, \lambda;s+1,c-s+1)\ge
\left\{
\begin{aligned}
 \left(\frac{s_0}{c}-\frac{1}{n}\right) \binom{n}{t}, & ~\text{ if }\lambda\ge s_0+1;\\
\left(\frac{c-s_0}{c}-\frac{1}{n}\right) \binom{n}{t}, & ~\text{ if }\lambda\ge c-s_0+1.
\end{aligned}
\right.
\]
If further $c\mid n$, we have
\[
m(n, t, \lambda;s+1,c-s+1)\ge
\left\{
\begin{aligned}
 \frac{s_0}{c}\binom{n}{t}, & ~\text{ if }\lambda\ge s_0+1;\\
\frac{c-s_0}{c}\binom{n}{t}, & ~\text{ if }\lambda\ge c-s_0+1.
\end{aligned}
\right.
\]
\end{corollary}

\begin{proof}
When $\lambda\ge s_0+1$, clearly $\lceil\lambda/s_0\rceil\ge 2$.
 If $\lceil\lambda/s\rceil\ge 2$, by Theorem \ref{thm-general-matching-lowerbound},
\begin{center}
$m(n,t,\lambda;s+1,c-s+1)\ge \binom{n}{t}-\binom{n-1}{t}= \frac{t}{n}\binom{n}{t}= \frac{\lceil sn/c\rceil}{n}\binom{n}{t}\ge \frac{s}{c}\binom{n}{t};$   \end{center}
if $\lceil\lambda/(c-s)\rceil\ge 2$, similarly
\begin{center}
$m(n,t,\lambda;s+1,c-s+1)\ge \binom{n}{t}-\binom{n-1}{n-t}= \frac{n-t}{n}\binom{n}{t}= \frac{n-\lceil sn/c\rceil}{n}\binom{n}{t}\ge (\frac{c-s}{c}-\frac{1}{n})\binom{n}{t}.$   \end{center}
Therefore,
\[m(n, t, \lambda;s+1,c-s+1)\ge \left(\frac{s_0}{c}-\frac{1}{n}\right) \binom{n}{t}.\]
Adding the condition $c\mid n$, obviously we can omit the ``$-\frac{1}{n}$'' factor in the above computation. Hence,
$m(n, t, \lambda;s+1,c-s+1)\ge \frac{s_0}{c}\binom{n}{t}$.

When $\lambda\ge c-s_0+1$, clearly $\lceil\frac{\lambda}{s}\rceil\ge 2$ and $\lceil\frac{\lambda}{c-s}\rceil\ge 2$.
The above argument tells that
\[m(n, t, \lambda;s+1,c-s+1)\ge \max\left\{\frac{s}{c},\frac{c-s}{c}-\frac{1}{n}\right\}\binom{n}{t}\ge \left(\frac{c-s_0}{c}-\frac{1}{n}\right) \binom{n}{t}.\]
Adding the condition $c\mid n$, similarly we have
$m(n, t, \lambda;s+1,c-s+1)\ge \frac{c-s_0}{c}\binom{n}{t}$.
\end{proof}

Note that we do not require $\lambda\equiv sn\pmod{c}$ in Corollary~\ref{cor-focal-lowerbound-explicit}.

\begin{corollary}\label{cor-gen-matchingnumber-exact}
Let $n,c,s,t$ be integers with $n,c\ge 2, s\in [c-1]$. Suppose that $ c\mid n$. Let $t=sn/c$ and $s_0=\min \{s,c-s\}$. Then 
\[m(n, t, s_0+1;s+1,c-s+1)= \frac{s_0}{c} \binom{n}{t},  ~~m(n, t, c-s_0+1;s+1,c-s+1)= \frac{c-s_0}{c} \binom{n}{t}.\]
\end{corollary}
\begin{proof}
The results immediately follow from Corollarys~\ref{cor-focal-upperbound-explicit} and~\ref{cor-focal-lowerbound-explicit}.
\end{proof}

Although we get two exact values of generalized matching numbers in Corollary~\ref{cor-gen-matchingnumber-exact}, the parameter $\lambda$ is restricted to $s_0+1$ or $c-s_0+1$, which does not match the condition of $\lambda$ in Theorems~\ref{mainthm-turan density} and~\ref{mainthm:code} for quantitative frameproof hypergraphs or codes, which require $\lambda=c$ when $ c\mid n$.

\section{Concluding remarks}\label{section-conclusion}
In this paper, we introduced quantitative frameproof codes and hypergraphs, which connect between frameproof codes (hypergraphs) and focal-free codes (hypergraphs). We presented asymptotically tight upper and lower bounds for both quantitative frameproof hypergraphs and codes, with the previous works of frameproof codes (hypergraphs) and focal-free codes (hypergraphs) as two extremal results.
We also determined the exact values of these hypergraphs and codes for infinitely many parameters.

Moreover, we introduced the generalized matching number as a key function in the above bounds. We also estimated upper and lower bounds on the  generalized matching number, and even characterized two exact values; see Corollary~\ref{cor-gen-matchingnumber-exact}.
However, many related problems remain open.
\begin{itemize}
  \item We have determined $f_{c,s}^{q}(n)$ asymptotically for fixed $c,s,n$ and $q\to \infty$. But how does $f_{c,s}^{q}(n)$ behave for fixed $c,s,q$ and $n\to \infty$? 
  \item Our estimates on the generalized matching number are very rough. Can we improve these bounds? 
\end{itemize}

\section*{Acknowledgements}
The research of Xiande Zhang is supported by the National Key
Research and Development Programs of China 2023YFA1010200 and 2020YFA0713100, the NSFC
under Grants No. 12171452 and No. 12231014, and the Innovation Program for Quantum Science
and Technology 2021ZD0302902.
The research of Xinqi Huang is supported by IBS-R029-C4.

{\small\bibliographystyle{abbrv}
\normalem
\bibliography{reference}}

@article{alon2023near,
  title={Near-sunflowers and focal families},
  author={Alon, Noga and Holzman, Ron},
  journal={Israel Journal of Mathematics},
  volume={256},
  number={1},
  pages={21--33},
  year={2023},
  publisher={Springer}
}

@incollection{frankl1987colored,
  title={Colored packing of sets},
  author={Frankl, P and F{\"u}redi, Z},
  booktitle={North-Holland Mathematics Studies},
  volume={149},
  pages={165--177},
  year={1987},
  publisher={Elsevier}
}

@article{blackburn2003frameproof,
  title={Frameproof codes},
  author={Blackburn, Simon R},
  journal={SIAM Journal on Discrete Mathematics},
  volume={16},
  number={3},
  pages={499--510},
  year={2003},
  publisher={SIAM}
}

@inproceedings{liu2024near,
  title={Near optimal constructions of frameproof codes},
  author={Liu, Miao and Ma, Zengjiao and Shangguan, Chong},
  booktitle={2024 IEEE International Symposium on Information Theory (ISIT)},
  pages={3154--3159},
  year={2024},
  organization={IEEE}
}

@book{jukna2011extremal,
  title={Extremal combinatorics: with applications in computer science},
  author={Jukna, Stasys},
  volume={571},
  year={2011},
  publisher={Springer}
}

@book{colbourn2010crc,
  title={CRC handbook of combinatorial designs},
  author={Colbourn, Charles J},
  year={2010},
  publisher={CRC press}
}

@article{keevash2014existence,
  title={The existence of designs},
  author={Keevash, Peter},
  journal={arXiv preprint arXiv:1401.3665},
  year={2014}
}

@article{rodl1985packing,
  title={On a packing and covering problem},
  author={R{\"o}dl, Vojt{\v{e}}ch},
  journal={European Journal of Combinatorics},
  volume={6},
  number={1},
  pages={69--78},
  year={1985},
  publisher={Elsevier}
}

@article{delcourt2024refined,
  title={Refined Absorption: A New Proof of the Existence Conjecture},
  author={Delcourt, Michelle and Postle, Luke},
  journal={arXiv preprint arXiv:2402.17855},
  year={2024}
}

@book{glock2023existence,
  title={The existence of designs via iterative absorption: hypergraph {F}-designs for arbitrary {F}},
  author={Glock, Stefan and K{\"u}hn, Daniela and Lo, Allan and Osthus, Deryk},
  volume={284},
  number={1406},
  year={2023},
  publisher={Memoirs of the American Mathematical Society}
}

@article{frankl1985near,
  title={Near perfect coverings in graphs and hypergraphs},
  author={Frankl, P{\'e}ter and R{\"o}dl, Vojtech},
  journal={European Journal of Combinatorics},
  volume={6},
  number={4},
  pages={317--326},
  year={1985},
  publisher={Elsevier}
}

@article{pippenger1989asymptotic,
  title={Asymptotic behavior of the chromatic index for hypergraphs},
  author={Pippenger, Nicholas and Spencer, Joel},
  journal={Journal of combinatorial theory, Series A},
  volume={51},
  number={1},
  pages={24--42},
  year={1989},
  publisher={Elsevier}
}

@article{huang2024focal,
  title={Focal-free uniform hypergraphs and codes},
  author={Huang, Xinqi and Shangguan, Chong and Zhang, Xiande and Zhao, Yuhao},
  journal={arXiv preprint arXiv:2410.23611},
  year={2024}
}

@article{boneh1998collusion,
  title={Collusion-secure fingerprinting for digital data},
  author={Boneh, Dan and Shaw, James},
  journal={IEEE transactions on information theory},
  volume={44},
  number={5},
  pages={1897--1905},
  year={1998},
  publisher={IEEE}
}

@article{staddon2001combinatorial,
  title={Combinatorial properties of frameproof and traceability codes},
  author={Staddon, Jessica N and Stinson, Douglas R and Wei, Ruizhong},
  journal={IEEE transactions on information theory},
  volume={47},
  number={3},
  pages={1042--1049},
  year={2001},
  publisher={IEEE}
}

@article{cheng2011separable,
  title={Separable codes},
  author={Cheng, Minquan and Ji, Lijun and Miao, Ying},
  journal={IEEE transactions on information theory},
  volume={58},
  number={3},
  pages={1791--1803},
  year={2011},
  publisher={IEEE}
}

@article{gao2014new,
  title={New bounds on separable codes for multimedia fingerprinting},
  author={Gao, Fei and Ge, Gennian},
  journal={IEEE transactions on information theory},
  volume={60},
  number={9},
  pages={5257--5262},
  year={2014},
  publisher={IEEE}
}

@article{blackburn2015probabilistic,
  title={Probabilistic existence results for separable codes},
  author={Blackburn, Simon R},
  journal={IEEE Transactions on Information Theory},
  volume={61},
  number={11},
  pages={5822--5827},
  year={2015},
  publisher={IEEE}
}

@article{shangguan2018new,
  title={New upper bounds for parent-identifying codes and traceability codes},
  author={Shangguan, Chong and Ma, Jingxue and Ge, Gennian},
  journal={Designs, Codes and Cryptography},
  volume={86},
  number={8},
  pages={1727--1737},
  year={2018},
  publisher={Springer}
}

@article{gu2019probabilistic,
  title={Probabilistic existence results for parent-identifying schemes},
  author={Gu, Yujie and Cheng, Minquan and Kabatiansky, Grigory and Miao, Ying},
  journal={IEEE Transactions on Information Theory},
  volume={65},
  number={10},
  pages={6160--6170},
  year={2019},
  publisher={IEEE}
}

@article{blackburn2010traceability,
  title={Traceability codes},
  author={Blackburn, Simon R and Etzion, Tuvi and Ng, Siaw-Lynn},
  journal={Journal of Combinatorial Theory, Series A},
  volume={117},
  number={8},
  pages={1049--1057},
  year={2010},
  publisher={Elsevier}
}

@article{gu2017bounds,
  title={Bounds on traceability schemes},
  author={Gu, Yujie and Miao, Ying},
  journal={IEEE Transactions on Information Theory},
  volume={64},
  number={5},
  pages={3450--3460},
  year={2017},
  publisher={IEEE}
}

@article{kabatiansky2019traceability,
  title={Traceability codes and their generalizations},
  author={Kabatiansky, Gregory A},
  journal={Problems of information transmission},
  volume={55},
  number={3},
  pages={283--294},
  year={2019},
  publisher={Springer}
}

@article{blackburn2008bound,
  title={A bound on the size of separating hash families},
  author={Blackburn, Simon R and Etzion, Tuvi and Stinson, Douglas R and Zaverucha, Gregory M},
  journal={Journal of Combinatorial Theory, Series A},
  volume={115},
  number={7},
  pages={1246--1256},
  year={2008},
  publisher={Elsevier}
}

@article{stinson2008generalized,
  title={On generalized separating hash families},
  author={Stinson, Douglas R and Wei, Ruizhong and Chen, Kejun},
  journal={Journal of Combinatorial Theory, Series A},
  volume={115},
  number={1},
  pages={105--120},
  year={2008},
  publisher={Elsevier}
}

@article{bazrafshan2011bounds,
  title={Bounds for separating hash families},
  author={Bazrafshan, Marjan and Van Trung, Tran},
  journal={Journal of Combinatorial Theory Series A},
  volume={118},
  number={3},
  pages={1129--1135},
  year={2011},
  publisher={Academic Press, Inc. Orlando, FL, USA}
}

@article{shangguan2016separating,
  title={Separating hash families: A {J}ohnson-type bound and new constructions},
  author={Shangguan, Chong and Ge, Gennian},
  journal={SIAM Journal on Discrete Mathematics},
  volume={30},
  number={4},
  pages={2243--2264},
  year={2016},
  publisher={SIAM}
}

@article{ge2019some,
  title={Some intriguing upper bounds for separating hash families},
  author={Ge, Gennian and Shangguan, Chong and Wang, Xin},
  journal={Science China Mathematics},
  volume={62},
  number={2},
  pages={269--282},
  year={2019},
  publisher={Springer}
}

@article{wei2025separating,
  title={Separating hash families with large universe},
  author={Wei, Xin and Zhang, Xiande and Ge, Gennian},
  journal={Journal of Combinatorial Theory, Series A},
  volume={216},
  pages={106075},
  year={2025},
  publisher={Elsevier}
}

@article{erdos1965problem,
  title={A problem on independent $r$-tuples},
  author={Erd\H{o}s, Paul},
  journal={Ann. Univ. Sci. Budapest. E{\"o}tv{\"o}s Sect. Math},
  volume={8},
  number={93-95},
  pages={2},
  year={1965}
}

@article{d2017cover,
  title={Cover-free codes and separating system codes},
  author={D'yachkov, Arkadii G and Vorobyev, Ilya V and Polyanskii, NA and Shchukin, V Yu},
  journal={Designs, Codes and Cryptography},
  volume={82},
  number={1},
  pages={197--209},
  year={2017},
  publisher={Springer}
}

@article{wei2023cover,
  title={On cover-free families},
  author={Wei, Ruizhong},
  journal={arXiv preprint arXiv:2303.17524},
  year={2023}
}

@article{liu2025near,
  title={Near Optimal Probabilistic Constructions of Frameproof Codes},
  author={Liu, Miao and Ma, Zengjiao and Shangguan, Chong},
  journal={IEEE Transactions on Information Theory},
  volume={71},
  number={6},
  pages={4137--4144},
  year={2025},
  publisher={IEEE}
}

@article{erdos1960intersection,
  title={Intersection theorems for systems of sets},
  author={Erd{\H{o}}s, Paul and Rado, Richard},
  journal={Journal of the London Mathematical Society},
  volume={1},
  number={1},
  pages={85--90},
  year={1960},
  publisher={Wiley Online Library}
}

@article{erdHos1978combinatorial,
  title={Combinatorial properties of systems of sets},
  author={Erd{\H{o}}s, Paul and Szemer{\'e}di, Endre},
  journal={J. Combinatorial Theory Ser. A},
  volume={24},
  number={3},
  pages={308--313},
  year={1978}
}

\appendix

\section{Proof of Theorem \ref{mainthm:code}}\label{sec:main-code}

In this section, we present proofs of Theorems \ref{thm-gen-code-upperbound} and \ref{thm-gen-code-lowerbound}.
We first connect subsets of $[q]^n$ with multi-partite hypergraphs as in \cite{huang2024focal}.

For positive integers $q, n$, let $\pi:[q]^n\rightarrow \binom{[n]\times [q]}{n}$ be an injective mapping defined by $$\pi(\bm x):=\{(1,x_1),(2,x_2),\dots,(n,x_n)\}$$ for each $\bm x\in [q]^n$. For $\mathcal{C}\subseteq [q]^n$, let $\pi(\mathcal{C}):=\{\pi(\bm x):\bm x\in\mathcal{C}\}$. Then let
$\mH_n(q)=\pi([q]^n)$ be the complete $n$-partite $n$-uniform hypergraph over $[n]\times [q]$ with a natural partition $V_i = \{i\}\times [q]$, $i\in [n]$. For any $t\leq n$, let $\mH_n^{(t)}(q)$ be the complete $n$-partite $t$-uniform hypergraph defined on the same vertex set and parts as $\mH_n(q)$. Clearly, $|\mH_n^{(t)}(q)|=\binom{n}{t}q^t$.
%
%
 Moreover, for a $t$-subset  $T\subseteq [n]$, let $\pi(\bm x_T)=\{(i,x_i):i\in T\}\in \mathcal{H}_n^{(t)}(q)$. Crucially, $\pi$ inherits the own subsequence property in the sense that $\bm x_T$ is an own subsequence of $\bm x$ with respect to $\mathcal{C}$ if and only if $\pi(\bm x_T)$ is an own subset of $\pi(\bm x)$ with respect to $\pi(\mathcal{C})$.

\subsection{Proof of Theorem \ref{thm-gen-code-upperbound}}

\begin{proof}[Proof of Theorem~\ref{thm-gen-code-upperbound}]
    Suppose that $\C\subseteq [q]^n$ is a $(c,s)$-frameproof code. Let $\C_0$ be the set of codewords in $\C$ that have no own $(t-1)$-subsequence with respect to $\C$ and $\mathcal{C}_1= \mathcal{C}\backslash{C}_0$. By definition of $\pi$, $\pi(\mathcal{C})\subset \mH_n(q)$ is a $(c,s)$-frameproof hypergraph and $\pi(\mathcal{C}_0)$ is the family of $\pi(\bm x)\in \pi(\mathcal{C})$ without own $(t-1)$-subsets with respect to $\pi(\mathcal{C})$. Then by Lemma \ref{lem-gen-hypergraph-ownsubset}, every $\pi(\bm x)\in \pi(\mathcal{C}_0)$ contains at least $\binom{n}{t}-m(n,t,\lambda; s+1,c-s+1)$ own $t$-subsets with respect to $\pi(\mathcal{C})$.
    Let
    \[\mathcal{T}_0:= \{T\in \mH_n^{(t)}(q): T \text{ is an own $t$-subset of some edge in $\pi(\mathcal{C}_0)$ with respect to } \pi(\mathcal{C})\},\]
    and let
    \[\mathcal{T}_1:= \{T\in \mH_n^{(t)}(q): T \text{ contains one own $(t-1)$-subset of some edge in } \pi(\mathcal{C}_1) \}.\]
By definition of own subsets, $\mathcal{T}_0$ and $\mathcal{T}_1$ are disjoint, and thus $|\mathcal{T}_0|+|\mathcal{T}_1|\le |\mH_n^{(t)}(q)|= \binom{n}{t}q^t$.
By definition of own subsets again, we have
\begin{align}\label{eq-T0-2}
    |\mathcal{T}_0|\ge |\pi(\mathcal{C}_0)|\left(\binom{n}{t}-m(n,t,\lambda; s+1,c-s+1)\right).
\end{align}
For $|\mathcal{T}_1|$, each edge in $\pi(\mathcal{C}_1)$ contains at least one own $(t-1)$-subset, and the latter generates $(n-t+1)q$ sets of $\mathcal{T}_1$. Meanwhile, each set of $\mathcal{T}_1$ is generated by at most $t$ own $(t-1)$-subsets, which implies that
\begin{align}\label{eq-T1-2}
    |\mathcal{T}_1|\ge |\pi(\mathcal{C}_1)|\frac{(n-t+1)q}{t}.
\end{align}

Combining (\ref{eq-T0-2}) and (\ref{eq-T1-2}) yields that
\begin{align*}
    \binom{n}{t}q^t&\ge |\mathcal{T}_0|+|\mathcal{T}_1|
    \ge |\pi(\mathcal{C}_0)|\left(\binom{n}{t}-m(n,t,\lambda; s+1,c-s+1)\right)+ |\pi(\mathcal{C}_1)|\frac{(n-t+1)q}{t}\\
    &\ge (|\pi(\mathcal{C}_0)|+|\pi(\mathcal{C}_1)|)\left(\binom{n}{t}-m(n,t,\lambda; s+1,c-s+1)\right)\\
    &= |\pi(\mathcal{C})|\left(\binom{n}{t}-m(n,t,\lambda; s+1,c-s+1)\right)
\end{align*}
as needed, where the last inequality holds when $q\ge \frac{t}{n-t+1}(\binom{n}{t}-m(n,t,\lambda;s+1,c-s+1))$.
\end{proof}

\subsection{Proof of Theorem \ref{thm-gen-code-lowerbound}}

To prove Theorem \ref{thm-gen-code-lowerbound}, we need a variant of  Lemma \ref{lem-induced packing}  for multi-partite hypergraphs.
Suppose $\mF$ and $\mH$ are $k$-partite hypergraphs with vertex partitions $V(\mF) = \cup_{i = 1}^k W_i$ and $V(\mH) = \cup_{i = 1}^k V_i$, respectively.
A copy $\mF'$ of $\mF$ in $\mH$ is called {\it faithful} \cite{liu2024near} if for each $i\in [k]$, the copy of $W_i$ in $V(\mF')$ is contained in $V_i$. An $\mF$-packing $\{(V(\mF_i),\mF_i):i\in[m]\}$ in $\mathcal{H}$ is said to be {\it faithful}, if for every $j\in [m]$, $\mF_j$ is a faithful copy of $\mF$.

\begin{lemma}[\cite{liu2024near}]\label{lem-faithful induced packing}
    Let $n>t$ and $\mathcal{F}\subseteq\binom{[n]}{t}$ be fixed. Then viewing $\mathcal{F}$ as an $n$-partite $t$-uniform hypergraph, there exists a faithful induced $\mF$-packing $\{(V(\mF_i),\mF_i):i\in[m]\}$ in $\mH_n^{(t)}(q)$ with $m\ge(1-o(1))\frac{\binom{n}{t} q^t}{|\mF|}$, where $o(1)\rightarrow 0$ as $q\rightarrow\infty$.
\end{lemma}

\begin{proof}[Proof of Theorem~\ref{thm-gen-code-lowerbound}]
 Let $\mG\subseteq\binom{[n]}{t}$ be one of the largest $t$-uniform hypergraphs on $n$ vertices that contains no $(s+1,c-s+1)$-disjoint collection of $\lambda$ edges, where $t=\lceil sn/c\rceil$. Then we have $|\mG|=m(n,t,\lambda;s+1,c-s+1)$. Let $\mF = \binom{[n]}{t}\backslash \mG$ be the graph on $[n]$, then
        $|\mF| = \binom{n}{t} - |\mG|
              = \binom{n}{t} - m(n,t,\lambda;s+1,c-s+1)$.
    Applying Lemma \ref{lem-faithful induced packing} with $\mF$, we get a faithful induced $\mF$-packing $\{(V(\mF_i),\mF_i): i\in[m]\}$ in $\mH_n^{(t)}(q)$ with
    \begin{align*}
        m\ge (1-o(1))\frac{\binom{n}{t}q^t}{|\mF|}
        =
        (1- o(1))\frac{\binom{n}{t}q^t}{\binom{n}{t} - m(n,t,\lambda;s+1,c-s+1)}.
    \end{align*}

    Let $\C=\{\pi^{-1}(V(\mF_i)): i\in[m]\}$.
    It remains to show that $\C$ is $(c,s)$-frameproof. Assume on the contrary that $\bm x,\bm x^1, \ldots, \bm x^c\in\C$ form a $(c,s)$-focal code with focus $\bm x$, that is, $s[n]\subseteq I(\bm x,\bm x^1)\uplus\cdots\uplus I(\bm x,\bm x^c)$. Then we can find $A_1\subseteq I(\bm x,\bm x^1),\ldots, A_c\subseteq I(\bm x,\bm x^c)$ such that $s[n]=A_1\uplus\cdots\uplus A_c$.
    By Observation \ref{obs-own subset} (ii),  the $c$ members $A_1,\ldots,A_c$ are $(s+1,c-s+1)$-disjoint subsets. Suppose that $\bm x=\pi^{-1}(V(\mathcal{F}_m))$ and for each $i\in[c]$, $\bm x^i=\pi^{-1}(V(\mathcal{F}_{j_i}))$ for some $j_i\neq m$. Then we have $|A_i|\le|I(\bm x,\bm x^i)|=|V(\mathcal{F}_m)\cap V(\mathcal{F}_{j_i})|\le t$ for each $i\in[c]$. So there are at least $\lambda$ distinct $i\in[c]$ such that $|A_i|=t$, which implies that $|I(\bm x,\bm x^i)|=|V(\mathcal{F}_m)\cap V(\mathcal{F}_{j_i})|=t$. By the same reason as in the proof of Theorem \ref{thm-gen-hypergraph-lowerbound}, we obtain a contradiction.
\end{proof}

\end{document}